\documentclass[11pt,letterpaper]{article}
\usepackage[T1]{fontenc}
\usepackage{lmodern}
\usepackage[margin=1in]{geometry}
\usepackage{amsmath,amssymb,amsthm,booktabs,array,tabularx,microtype,flafter}
\usepackage[colorlinks=true,linkcolor=blue,citecolor=blue,urlcolor=blue]{hyperref}
\theoremstyle{plain}
\newtheorem{theorem}{Theorem}[section]
\newtheorem{proposition}[theorem]{Proposition}
\newtheorem{lemma}[theorem]{Lemma}
\newtheorem{corollary}[theorem]{Corollary}
\theoremstyle{definition}
\newtheorem{definition}[theorem]{Definition}

\newtheorem{remark}[theorem]{Remark}
\newcommand{\F}{\mathbb F}
\newcommand{\A}{\mathcal A}
\newcommand{\mc}[1]{\mathcal{#1}}
\newcommand{\dist}{\operatorname{d}}
\newcommand{\wt}{\operatorname{wt}}
\title{Length-Maximal Codes over Arbitrary Alphabets:\\Structure, Distance Distributions, and Nonexistence}

\author{Tim Alderson\thanks{Department of Mathematics and Statistics, University of New Brunswick Saint John, Saint John, NB, Canada. The author acknowledges the support of the Natural Sciences and Engineering Research Council of Canada (NSERC), funding reference number RGPIN-2019-04103. Cette recherche a \'et\'e financ\'ee par le Conseil de recherches en sciences naturelles et en g\'enie du Canada (CRSNG), num\'ero de r\'ef\'erence RGPIN-2019-04103.}}

\date{}

\begin{document}
	\maketitle
	\begin{abstract}
		An $(n,q^k,d)_q$ code of integer dimension $k\ge2$ and Singleton defect $s=n-k+1-d$ satisfies $n\le(s+1)(q+1)+k-2$. We study codes attaining this shortening--Plotkin bound, without assuming linearity, additivity, or systematicity. Equality forces an orthogonal array of strength at least $k-1$ and restricts the pairwise distances to $\{d\}\cup\{n-k+3,\ldots,n\}$. We determine the distance distribution about every codeword and show that the minimum-distance agreement sets form a $(k-1)$-design. For $k\ge3$, the defect satisfies $s\le q-1$, and $(s+2)\mid q$ whenever $s\le q-2$. In dimension four, every positive-defect code has $s=q-2$, except for the binary extended Hamming code. In dimensions at least five, the only length-maximal codes of positive defect, up to equivalence, are the dual ternary Golay code and the extended ternary Golay code. Classical restrictions on MDS codes then exclude all length-maximal codes in these dimensions over prime-power alphabets of size at least four. In dimension four we determine exactly which defects occur over every prime-power alphabet, and in dimension three we do so over even prime-power and odd prime alphabets. These results extend the additive nonexistence theory to unrestricted codes. Every length-maximal code attains the numerical Griesmer bound; further shortening arguments give sufficient conditions for arbitrary codes to satisfy that bound, including $d\le q^2$ and, in the binary case, $d\le8$.
	\end{abstract}
	\noindent\textbf{Keywords:} nonlinear codes; Singleton defect; orthogonal arrays; distance distributions; resolvable designs; Griesmer bound.
	
	\smallskip
	\noindent\textbf{MSC 2020:} 94B65; 94B25; 05B05; 05B15.
	
	\section{Introduction}
	
	An $(n,M,d)_q$ code is a set of $M$ words of length $n$ over an alphabet of size $q$, with minimum Hamming distance $d$. Throughout, $M=q^k$ for an integer $k\ge2$. We call $k$ the dimension, without assuming any algebraic structure on the code. Its Singleton defect is
	\[
	s=n-k+1-d\ge0.
	\]
	A code of defect $s$ is called an A$^s$MDS code. The cases $s=0$ and $s=1$ are the maximum distance separable (MDS) and almost MDS (AMDS) codes, respectively \cite{Singleton1964,DeBoer1996}.
	
	For fixed $q$, $k$, and $s$, repeated shortening followed by the Plotkin bound gives
	\begin{equation}\label{eq:arc_bound}
		n\le(s+1)(q+1)+k-2.
	\end{equation}
	We call a code attaining equality in \eqref{eq:arc_bound} \emph{length-maximal}. Chen et al.\ \cite{ChenEtAl2024} have obtained bounds for general codes of small Singleton defect using covering codes. Here we study what equality in \eqref{eq:arc_bound} forces over an arbitrary alphabet.
	
	For linear codes, equality in \eqref{eq:arc_bound} may be interpreted geometrically as follows: the columns of a generator matrix of a linear $[n,k,d]_q$ code of full support, viewed as projective points and counted with multiplicity, form an $(n,r)$-multi-arc in $\mathrm{PG}(k-1,q)$, where $r=n-d=s+k-1$. Alderson and Zhang \cite[p.~101]{alderson2025projectivesystems}, define a linear code to be length-maximal when the corresponding  multi-arc is maximal, that is when
	\[
	n=(r-k+2)(q+1)+k-2,
	\]
	which is precisely the condition of equality in \eqref{eq:arc_bound}. When $k=3$, the ambient space is the projective plane, $r=s+2$, and \eqref{eq:arc_bound} reduces to $n\le(r-1)q+r$. Equality gives the classical planar maximal arcs introduced by Barlotti~\cite{Barlotti1956}. For $1<r<q$, Ball, Blokhuis, and Mazzocca \cite[p.~31]{Ball1997MaximalArcs} proved that no maximal arc of degree $r$ exists in $\mathrm{PG}(2,q)$ for odd $q$. For even $q$, Denniston's construction gives examples whenever $r\mid q$ \cite{Denniston1969}.
	
	Passing from linear to additive codes changes the picture, and already at $k=3$. The perp-system of De Clerck, Delanote, Hamilton, and Mathon \cite{DeClerck2002PerpSystems} gives an additive $(21,9^3,18)_9$ AMDS code \cite[Example~10]{alderson2026setssubspacesrestrictedhyperplane}. It is length-maximal, although a linear code with these parameters does not exist (such would give a maximal arc of degree three in $\mathrm{PG}(2,9)$). In higher dimensions, strong geometric restrictions persist. Representing additive codes by sets of projective subspaces, Alderson and Ball \cite[Theorems~1, 7, and~8; Remark~9]{alderson2026setssubspacesrestrictedhyperplane} proved that, for alphabet size $q>3$, a length-maximal additive code has dimension at most four. This settles the conjecture for linear codes in \cite[Conjecture~5.1]{alderson2025projectivesystems}. In dimension four, for $q>2$, the defect of a linear length-maximal code must be $q-2$, and ovoids provide examples.
	
	Which of these restrictions persist over an arbitrary alphabet? Our starting point is that equality in \eqref{eq:arc_bound} forces every symbol to occur equally often in each coordinate. This remains true after any sequence of shortenings that leaves dimension at least two. It follows that the codewords form an orthogonal array of strength at least $k-1$. Shortening also restricts the possible distances between codewords. Together, these two facts allow us to determine the distance distribution about each codeword.
	
	For three-dimensional length-maximal codes, the only possible nonzero distances are $d$ and $n$. Unrestricted two-distance codes, including this case, were studied by Boyvalenkov, Delchev, and the two Zinovievs 	\cite{BoyvalenkovEtAl2021,BoyvalenkovEtAl2022}. We use a spectral divisibility argument in this case to restrict the defect. In higher 	dimensions, Rao's bound \cite{Rao1947} restricts the parameters of the orthogonal array. Noda's classification of tight arrays of strengths three and five \cite{Noda1986} gives a second way to exclude one dimension-four case and places the extended ternary Golay example among the tight arrays of strength five. Related work of Gavrilyuk and Suda \cite{GavrilyukSuda2025} connects extremal orthogonal arrays with association schemes and Golay codes. We give elementary proofs of the specializations needed here. Our main results are as follows.
	\begin{enumerate}
		\item \textbf{Equality structure and distributions.} Every length-maximal code is an $\operatorname{OA}(q^k,n,q,k-1)$ of index $q$. The distances between distinct codewords belong to $\{d\}\cup\{n-j:0\le j\le k-3\}$. An explicit triangular system determines the distance distribution from $q$, $k$, and $s$, and the distribution is the same about every codeword (Theorem~\ref{thm:distribution}). For each fixed codeword $x$, the sets of coordinates on which its minimum-distance neighbours agree with $x$ form a $(k-1)$-$(n,s+k-1,q-1)$ design, with blocks counted with multiplicity (Corollary~\ref{thm:zero_set_design}).
		
		\item \textbf{Dimension restrictions.} 	For $q\ge4$, a length-maximal code of dimension $k\ge5$ must be MDS. In this case Bruen and Silverman's restriction 	\cite[Theorems~4 and~6]{BruenSilverman1983} requires $36\mid q$. Thus no such code exists over a prime-power alphabet of size at least four. For $s\ge1$, the only length-maximal codes of dimension at least five, up to equivalence, are the dual and extended ternary Golay codes, of dimensions five and six, respectively (Theorems~\ref{thm:dim5existence} 	and~\ref{thm:high_classification}).
		
		\item \textbf{Divisibility and dimension four.} For $k\ge3$, one has $s\le q-1$, and $(s+2)\mid q$ whenever $s\le q-2$ (Lemma~\ref{lem:divisibility}). In dimension four, every positive-defect code has $s=q-2$ except for the binary 		$(8,16,4)_2$ code, which is unique up to equivalence and is the extended Hamming code (Corollary~\ref{cor:dim4_trichotomy} and Proposition~\ref{prop:binary_unique}). When $k=4$ and $s=q-2$, no two codewords are at full distance, and ovoid codes give examples over every finite field. Together with known constructions, these restrictions determine exactly which defects occur in dimension four over prime-power alphabets, and in dimension three over even prime-power or odd prime alphabets (Corollary~\ref{cor:prime_power_existence}).
	\end{enumerate}
	
	Shortening also connects the length problem to the Griesmer bound. Every length-maximal code attains the numerical Griesmer bound (Corollary~\ref{cor:lm_griesmer_equality}). For general codes, Guerrini, Meneghetti, and Sala \cite{Guerrini2015Optimal}, following Bellini and Meneghetti \cite{Bellini2015Griesmer}, studied when this linear bound remains valid for systematic or unrestricted codes. We show that an unrestricted code of size $q^k$ satisfies the Griesmer inequality whenever $d\le q^2$. In the binary case, shortening to eight words and applying the result of Guerrini, Meneghetti, and Sala \cite[Corollary~32]{Guerrini2015Optimal} gives the range $d\le8$.
	
	Section~\ref{sec:prelim} develops shortening and the length bound. Section~\ref{sec:structure} proves the equality structure, the distance distribution theorem, and the design consequences. Section~\ref{sec:dimensions} derives the dimension restrictions and treats the exceptional codes. Section~\ref{sec:refinements} gives further bounds and Griesmer applications. Section~\ref{sec:conclusion} provides the remaining existence and extension questions.

	\section{Shortening and the length bound}\label{sec:prelim}\label{sec:maximal_arc_bound}
	\subsection{Conventions and shortening}
	Let $\A$ be an alphabet of size $q\ge2$ with a distinguished symbol $0$. For $x,y\in\A^n$, write $\dist(x,y)$ for their Hamming distance. Two codes are \emph{equivalent} if one can be obtained from the other by permuting the coordinates and independently permuting the alphabet in each coordinate \cite[Chapter~4]{Welsh1988}. These operations preserve all pairwise distances. In particular, by choosing the symbol permutations appropriately, we may send any prescribed codeword to $\mathbf{0}$.
	
	We call a code \emph{equivalent to linear} if $q$ is a prime power and, after identifying $\A$ with $\F_q$, its equivalence class contains an $\F_q$-linear code. All other codes are called \emph{genuinely nonlinear}. In particular, every code over an alphabet of non-prime-power size is genuinely nonlinear in this sense. When $\A=\F_q$ has characteristic $p$, an \emph{additive} code $\mathcal C$ is an $\F_p$-linear subspace of $\F_q^n$. As throughout the paper, its dimension is $\log_q|\mathcal C|$.
	
	A code $\mathcal C$ is \emph{symbol-uniform} if each symbol occurs exactly $|\mathcal C|/q$ times in each coordinate. For a coordinate $i$ and a symbol $a\in\A$, the \emph{shortened code} $\mathcal C_i^{(a)}$ is obtained by retaining all codewords $c$ with $c_i=a$ and deleting coordinate $i$ from each. Since the retained words agree in coordinate $i$, deleting it preserves their pairwise distances. Thus shortening cannot decrease the minimum distance whenever at least two words remain. 
	
	\begin{lemma}[Shortening]\label{lem:shortening}
		Let $\mathcal C$ have parameters $(n,q^k,d)_q$, with $k\ge2$. In any coordinate, some shortened code contains a subset of size $q^{k-1}$ and minimum distance at least $d$. Iterating, for $1\le r\le k$ one obtains a code of length $n-k+r$, size $q^r$, and minimum distance at least $d$. If $\mathcal C$ is symbol-uniform, then every one-coordinate shortening has size $q^{k-1}$.
	\end{lemma}
	\begin{proof}
		The $q$ sets of codewords obtained by fixing the symbol in a given coordinate partition the code, so at least one of these sets contains $q^{k-1}$ codewords. Select $q^{k-1}$ words from that set and delete the fixed coordinate. Distinct words remain distinct and all distances are preserved. The other assertions follow by iteration and the definition of symbol-uniformity.
	\end{proof}
	
	The Singleton bound \cite{Singleton1964} gives $s=n-k+1-d\ge0$. For every $(n,M,d)_q$ code, the length form of the Plotkin bound \cite{MR238597} gives
	\begin{equation}\label{eq:plotkin_length}
		n\ge\left\lceil d\,\frac{1-1/M}{1-1/q}\right\rceil.
	\end{equation}
	Plotkin \cite{Plotkin1960} treated the binary case. For arbitrary alphabets, Silverman \cite[Theorem~4.2]{SilvermanMetrization1960} gives the corresponding bound on code size and characterizes equality. Equality in \eqref{eq:plotkin_length} before rounding holds precisely when the code is equidistant and symbol-uniform.
	
	\subsection{The general length bound and equality in dimension two}
	We begin in dimension two, where counting agreements between pairs of codewords gives both the length bound and its equality conditions. The general length bound then follows by shortening to this case.
	
	\begin{lemma}\label{lem:2d_bound}
		An $(n,q^2,d)_q$ code of defect $s$ satisfies $n\le(s+1)(q+1)$. Equality holds if and only if the code is symbol-uniform and equidistant.
	\end{lemma}
	\begin{proof}
		Let $f_{i,a}$ count the words carrying symbol $a$ in coordinate $i$. Counting agreements of unordered pairs of distinct words gives
		\begin{equation}\label{eqn:2d_double_count}
			nq\binom q2\ \le\ \sum_{i=1}^n\sum_{a\in\A}\binom{f_{i,a}}2
			\ \le\ \binom{q^2}2(n-d)=\binom{q^2}2(s+1).
		\end{equation}
		The first inequality follows from $\sum_a f_{i,a}=q^2$, with equality precisely when every $f_{i,a}=q$. The second is equality precisely when every distinct pair agrees in $n-d=s+1$ coordinates. Simplification gives the bound and its equality characterisation.
	\end{proof}
	
	\begin{theorem}[Length-Maximal bound]\label{thm:main_bound}
		Every $(n,q^k,d)_q$ code of defect $s$ and integer dimension $k\ge2$ satisfies
		\begin{equation}\label{eqn:maximal_arc_bound}
			n\le(s+1)(q+1)+k-2.
		\end{equation}
		Equivalently,
		\begin{equation}\label{eqn:singleton_to_griesmer}
			n\ge d+\left\lceil\frac dq\right\rceil+k-2.
		\end{equation}
	\end{theorem}
	\begin{proof}
		Shorten to a code of length $N=n-k+2$, size $q^2$, and minimum distance $d'\ge d$. Equation~\eqref{eq:plotkin_length} gives $N\ge\lceil d'(q+1)/q\rceil\ge d+\lceil d/q\rceil$. This proves \eqref{eqn:singleton_to_griesmer}. Substituting $d=n-k+1-s$ gives \eqref{eqn:maximal_arc_bound}.
	\end{proof}
	
	\begin{definition}\label{def:length_maximal}
		A code attaining \eqref{eqn:maximal_arc_bound} is \emph{length-maximal}. Its parameters satisfy
		\[
		n=(s+1)(q+1)+k-2,\qquad d=q(s+1).
		\]
	\end{definition}
	
	The following observation will ensure that shortening preserves both the defect and the minimum distance in the equality arguments.
	
	\begin{lemma}[Defect Preservation]\label{lem:defect_preserved}
		Let $\mathcal{C}$ be an $(n, q^k, d)_q$ A$^s$MDS code with $k \ge 3$, and let $\mathcal{C}'$ be an $(n-1, q^{k-1}, d')_q$ A$^{s'}$MDS code contained in a shortened code $\mathcal{C}_i^{(a)}$ (in particular, $\mathcal{C}'$ may be $\mathcal{C}_i^{(a)}$). If \[
		n \ge s(q+1) + k - 1,
		\]
		then $s' = s$ and $d' = d$.
	\end{lemma}
	
	\begin{proof}
		Since $\mathcal{C}'$ inherits minimum distance $d' \ge d$ from $\mathcal{C}$, we have $s' = n - k + 1 - d' \le s$. Suppose for contradiction that $s' \le s-1$. Applying Theorem~\ref{thm:main_bound} to $\mathcal{C}'$ gives
		\[
		n - 1 \le (s'+1)(q+1) + (k-1) - 2 \le s(q+1) + k - 3,
		\]
		so $n \le s(q+1)+k-2$, contradicting the hypothesis. Hence $s'=s$, and $d' = (n-1)-s-(k-1)+1 = n-s-k+1 = d$.
	\end{proof}
	
	\subsection{Dimension two and resolvable designs}\label{sec:resolvable_correspondence}

	In dimension two, symbol-uniformity turns each coordinate into a partition of the codewords, while equidistance determines how often two codewords lie in the same part.
	
	For general background on designs and resolvability, see \cite{BethJungnickelLenz1999}. Here we allow repeated blocks and use the term \emph{multidesign} to make this convention explicit. A $2$-$(v,K,\lambda)$ multidesign on a set $V$ of $v$ points is a multiset of $K$-subsets of $V$, called \emph{blocks}, such that each pair of distinct points lies in exactly $\lambda$ blocks, counted with multiplicity. A \emph{parallel class} is a collection of block occurrences that partitions $V$. A \emph{resolution} is a partition of all block occurrences into parallel classes, and a multidesign admitting a resolution is \emph{resolvable}. Blocks may repeat, and distinct parallel classes may induce the same partition of $V$. Throughout, we equip each resolvable multidesign with a chosen resolution. An isomorphism is a bijection of point sets that carries the parallel classes of one chosen resolution to those of the other, with multiplicities preserved.
	
	The following proposition specializes the Semakov--Zinov'ev correspondence \cite[Theorem~1]{SemakovZinoviev1968} to dimension two. Silverman \cite[Theorem~4.2]{SilvermanMetrization1960} had earlier described a closely related correspondence for the equality case of the Plotkin bound.
	
	\begin{proposition}\label{prop:2d_design_equiv}
		For integers $q\ge2$ and $s\ge0$, the equivalence classes of length-maximal $(n,q^2,d)_q$ A$^s$MDS codes are in bijection with the isomorphism classes of resolvable $2$-$(q^2,q,s+1)$ multidesigns equipped with a chosen resolution. The number of parallel classes equals the code length $n=(s+1)(q+1)$.
	\end{proposition}
	
	\begin{proof}
		Let $\mc{C}$ be such a code. By Lemma~\ref{lem:2d_bound}, $\mc{C}$ is symbol-uniform and equidistant, with every pair of codewords agreeing in exactly $s+1$ coordinates. Take the $q^2$ codewords as points. For a coordinate $i$ and a symbol $a\in\A$, put	$B_{i,a}=\{c\in\mc C:c_i=a\}$, and $\mathcal P_i=\{B_{i,a}:a\in\A\}$.
		
		Symbol-uniformity gives $|B_{i,a}|=q$, so $\mathcal P_i$ is a parallel class of $q$ blocks. Two codewords lie in a common block of $\mathcal P_i$ if and only if they agree in coordinate $i$. Thus every pair of points lies in exactly $s+1$ blocks. These blocks form a $2$-$(q^2,q,s+1)$ multidesign, and the classes $\mathcal P_1,\ldots,\mathcal P_n$ give a resolution.
		
		Conversely, let a resolvable $2$-$(q^2,q,s+1)$ multidesign have parallel classes $\mathcal P_1,\ldots,\mathcal P_n$. Its replication number is
		\[
		r=\frac{(s+1)(q^2-1)}{q-1}=(s+1)(q+1).
		\]
		Each point lies in one block of every parallel class, so $n=r$. For each $i$, label the $q$ blocks of $\mathcal P_i$ bijectively by the symbols of $\A$. For a point $x\in V$, define the word $c(x)$ by letting $c(x)_i$ be the label of the unique block of $\mathcal P_i$ containing $x$. If $x\ne y$, then $c(x)$ and $c(y)$ agree in exactly $s+1<n$ coordinates. Hence $x\mapsto c(x)$ is injective and produces $q^2$ words of mutual distance 	$d=n-(s+1)=q(s+1)$. Every block has size $q$, so each symbol occurs $q$ times in every coordinate. The resulting code is therefore symbol-uniform, equidistant, and length-maximal. Finally, a point bijection merely relabels the codewords, while permuting the parallel classes or relabelling the blocks within a class corresponds, respectively, to a coordinate permutation or a symbol permutation. Applying the two constructions successively recovers the original code or multidesign, with its resolution, up to these equivalences. This proves the correspondence.
	\end{proof}
	
	\begin{corollary}\label{cor:2d_existence}
		Let $q \ge 2$ and $s \ge 0$.
		\begin{enumerate}
			\item For \(s=0\), a length-maximal code of dimension two exists if and only if an affine plane of order \(q\) exists. Moreover, the code is equivalent to a linear code if and only if the corresponding affine plane is Desarguesian. 
			\item If $q$ is a prime power, a length-maximal $(n, q^2, d)_q$ A$^s$MDS code exists for every $s \ge 0$.
			\item If a non-Desarguesian affine plane of order $q$ exists, then a \emph{genuinely nonlinear} length-maximal $(n, q^2, d)_q$ A$^s$MDS code exists for every $s \ge 0$. The smallest order of a non-Desarguesian affine plane is \(9\).
		\end{enumerate}
	\end{corollary}
	
	\begin{proof}
		(1) A resolvable $2$-$(q^2,q,1)$ design is precisely an affine plane of order $q$, so the first assertion follows from Proposition~\ref{prop:2d_design_equiv}. Under this correspondence, the associated code is an index-$1$ net code based on that affine plane. By \cite[Theorem~2.5]{AldersonNetCodes2008}, such a code is equivalent to a linear code if and only if its underlying net is Desarguesian. Since the net has degree $q+1$, this is equivalent to the affine plane being isomorphic to $\mathrm{AG}(2,q)$.\\
		(2) For prime-power $q$ an affine plane of order $q$ exists, giving a resolvable $2$-$(q^2, q, 1)$ design. Repeating each of its parallel classes $s+1$ times produces a resolvable $2$-$(q^2, q, s+1)$ multidesign, equivalent by Proposition~\ref{prop:2d_design_equiv} to the required code. \\
		(3) Let $\pi$ be a non-Desarguesian affine plane of order $q$, and apply the construction of~(2), repeating each parallel class $s+1$ times. By Proposition~\ref{prop:2d_design_equiv}, the resulting multidesign corresponds to a length-maximal code $\mc C$. In the terminology of \cite{AldersonNetCodes2008}, $\mc C$ is a net code of uniform index $s+1$ based on $\pi$; repetition does not change its underlying net \cite[Lemma~2.1]{AldersonNetCodes2008}.
		
		If $\mc C$ were equivalent to a linear code, then it would be a Desarguesian net code by \cite[Theorem~2.5]{AldersonNetCodes2008}. Since equivalent net codes are based on isomorphic nets \cite[Theorem~2.4]{AldersonNetCodes2008}, this would imply $\pi\cong\mathrm{AG}(2,q)$, contrary to the choice of $\pi$. Hence $\mc C$ is not equivalent to a linear code.
		
		The least order of a non-Desarguesian affine plane is $9$. The first example was constructed by Veblen and Maclagan-Wedderburn \cite{VeblenMaclaganWedderburn1907}; minimality follows from the classification through order 8, completed by Hall, Swift, and Walker \cite{HallSwiftWalker1956}.
	\end{proof}

	\begin{remark}\label{rem:2d_open}
		If an affine plane of order $q$ exists, then repeating each of its parallel classes $s+1$ times gives a length-maximal code of dimension two and defect $s$.  When no such plane exists, a  different construction is therefore required. For example, although there is no affine plane of order \(6\), a resolvable \(2\)-\((36,6,2)\) multidesign exists by Jungnickel’s Example 6.1(d) \cite{JUNGNICKEL1990177}. Hence Proposition~\(\ref{prop:2d_design_equiv}\) yields a length-maximal \((14,36,12)_6\) AMDS code.
	\end{remark}
	
	\section{Equality structure and distance distributions}\label{sec:structure}
	\subsection{Uniformity and possible distances}
	An $N\times n$ array over $\A$ is an $\operatorname{OA}(N,n,q,r)$ if every ordered $r$-tuple occurs exactly $N/q^r$ times in each choice of $r$ columns \cite{BoseBush1952,HedayatSloaneStufken1999}. When a code is viewed as such an array, there is one row for each codeword and one column for each coordinate. Thus strength one is symbol-uniformity.
	
	\begin{theorem}[Symbol-uniformity under shortening]
		\label{thm:optimal_uniform}
		If $\mc{C}$ is a length-maximal $(n, q^k, d)_q$ A$^s$MDS code with $k \ge 2$, then $\mc{C}$ is symbol-uniform. Moreover, every code of dimension $r \ge 2$ obtained from $\mathcal{C}$ by successive shortening is a length-maximal $(n-k+r,\, q^r,\, d)_q$ A$^s$MDS code, and is therefore symbol-uniform.
	\end{theorem}
	
	\begin{proof}
		We prove by induction on $j\ge2$ that every length-maximal code of dimension $j$ is symbol-uniform. The case $j=2$ is the equality characterisation in Lemma~\ref{lem:2d_bound}.
		
		Let $j\ge3$, assume the assertion in dimension $j-1$, and let $\mathcal C$ be length-maximal of dimension $j$. Fix a coordinate $i$, and let $f_{i,a}$ be the number of words carrying $a\in\A$ in that coordinate. Write $F_{i,a}=\{c\in\mathcal C:c_i=a\}$, and choose $\alpha$ for which $f_{i,\alpha}=|F_{i,\alpha}|$ is largest. Thus $\mathcal C_i^{(a)}$ is obtained from $F_{i,a}$ by deleting coordinate $i$. Since the $q$ fibre sizes sum to $q^j$, we have $f_{i,\alpha}\ge q^{j-1}$.
		
		Delete coordinate $i$ from any $q^{j-1}$-subset of $F_{i,\alpha}$. The resulting code has length $n-1$ and dimension $j-1$. Since $ n=(s+1)(q+1)+j-2\ge s(q+1)+j-1, $ Lemma~\ref{lem:defect_preserved} shows that such a subset has defect $s$ and is length-maximal. It is therefore symbol-uniform by the induction hypothesis.
		
		Suppose that $f_{i,\alpha}>q^{j-1}$. The fibre then contains distinct words $u,v$ and a set $S$ of $q^{j-1}-1$ further words. The two codes $S\cup\{u\}$ and $S\cup\{v\}$ are both symbol-uniform after coordinate $i$ is deleted. Their symbol counts agree in every remaining coordinate, so $u$ and $v$ carry the same symbol in every remaining coordinate. They also agree in coordinate $i$, contrary to $u\ne v$. Thus $f_{i,\alpha}\le q^{j-1}$. Maximality of $f_{i,\alpha}$ and $\sum_a f_{i,a}=q^j$ now give $f_{i,a}=q^{j-1}$ for every $a\in\A$. Since $i$ was arbitrary, $\mathcal C$ is symbol-uniform. This completes the induction, and taking $j=k$ proves the first assertion.
		
		For the second assertion, let $\mathcal{C}$ be length-maximal of dimension $k \ge 3$; for $k = 2$ there is nothing to prove. By the first assertion, every full one-coordinate shortening of $\mathcal{C}$ has exactly $q^{k-1}$ words. Since $n \ge s(q+1)+k-1$, Lemma~\ref{lem:defect_preserved} shows that it has defect $s$ and minimum distance $d$. Its length is $n-1 = (s+1)(q+1)+(k-1)-2$, so it is a length-maximal $(n-1,q^{k-1},d)_q$ code, and hence symbol-uniform by the first assertion. Repeating this step $k-r$ times proves the assertion for every successive shortening of dimension $r \ge 2$.
	\end{proof}
	
	\begin{corollary}[Orthogonal-array structure]\label{cor:orthogonal_array}
		Every length-maximal $(n,q^k,d)_q$ A$^s$MDS code with $k\ge 2$ is an orthogonal array 	$\operatorname{OA}(q^k,n,q,k-1)$ of index $q$, with one row for each codeword and one column for each coordinate. More explicitly, if $1\le j\le k-1$, then any prescribed symbols in any $j$ prescribed coordinates occur together in exactly $q^{k-j}$ codewords.
	\end{corollary}
	
	\begin{proof}
		Fix $j$ coordinates and prescribed symbols. Shorten successively in the first $j-1$ of those coordinates. The resulting code has dimension $k-j+1\ge 2$ and is symbol-uniform by Theorem~\ref{thm:optimal_uniform}; hence the prescribed symbol in the last coordinate occurs $q^{k-j}$ times. This is the defining condition for an orthogonal array of strength $k-1$ and index $q$.
	\end{proof}
	
	\begin{theorem}[Possible distances]
		\label{thm:optimal_weights}
		If $\mathcal{C}$ is a length-maximal $(n, q^k, d)_q$ A$^s$MDS code with $k \ge 2$, then for any two distinct codewords $u, v \in \mathcal{C}$,	$\dist(u,v) \in \{d\} \cup \{n-j:0\le j\le k-3\}$.
	\end{theorem}
	
	\begin{proof}
		We proceed by induction on $k$. For $k=2$, the result follows from Lemma~\ref{lem:2d_bound}.

		Assume $k \ge 3$. By Theorem~\ref{thm:optimal_uniform}, $\mathcal{C}$ is symbol-uniform. Let $u\ne v \in \mathcal{C}$. If $u$ and $v$ agree in no coordinate then $\dist(u,v) = n$. Otherwise, let $i$ be a coordinate in which $u_i = v_i$. The shortened code $\mathcal{C}_i^{(u_i)}$ contains both $u'$ and $v'$ (the shortenings of $u$ and $v$, respectively). By Theorem~\ref{thm:optimal_uniform}, $\mathcal{C}_i^{(u_i)}$ is a length-maximal $(n-1, q^{k-1}, d)_q$ A$^s$MDS code, and the induction hypothesis gives
		\[
		\dist(u,v) = \dist(u',v') \in \{d\} \cup \{n-1,\ldots,(n-1)-(k-1)+3\}
		= \{d\} \cup \{n-1,n-2,\ldots,n-k+3\},
		\]
		completing the induction.
	\end{proof}
	
	For $x\in\mathcal C$, let
	
	\begin{equation}\label{eqn: dist_distrib}
		A_i(x)=|\{c\in\mathcal C:\dist(x,c)=i\}|.
	\end{equation}
	
	Under coordinatewise symbol permutations sending $x$ to $\mathbf0$, each codeword $c$ acquires weight $\dist(x,c)$. Thus the $A_i(x)$ are the coefficients of the Hamming weight enumerator of the relabelled code. We now show that these coefficients are independent of $x$.
	
	\subsection{The distribution theorem}
	By Delsarte's regularity theorem \cite[Theorem~5.24, p.~80]{Delsarte1973}, a design in the Hamming scheme with $\ell$ distinct nonzero distances is distance-invariant if its strength is at least $\ell-1$. A length-maximal code has strength at least $k-1$ and at most $k-1$ distinct nonzero distances, so the theorem applies. The following moment calculation gives the common distance distribution explicitly.
	\begin{theorem}[Distance distributions]\label{thm:distribution}\label{prop:completeness}\label{prop:min_weight_count}
		Let $\mathcal C$ be length-maximal with parameters $(n,q^k,d)_q$ and defect $s$. Put $t=n-d=s+k-1$. The distribution $A_i(x)$ given in \eqref{eqn: dist_distrib} is determined by $q,k,s$ and is the same for every $x\in\mathcal C$. In particular,
		\begin{equation}\label{eq:minimum_count}
			A_d=(q^2-1)\frac{\binom n{k-2}}{\binom t{k-2}}.
		\end{equation}
		For $k\ge3$, write $B_j=A_{n-j}$ for $0\le j\le k-3$. The remaining nonzero-distance counts are determined recursively by
		\begin{equation}\label{eq:triangular}
			B_j=\binom nj(q^{k-j}-1)-\binom tj A_d
			-\sum_{\ell=j+1}^{k-3}\binom\ell j B_\ell,
			\qquad j=k-3,\ldots,0.
		\end{equation}
		All other $A_i$ vanish, apart from $A_0=1$. For $k=2$, the only nonzero distance is $d$, with $A_d=q^2-1$.
	\end{theorem}
	\begin{proof}
		Fix $x$, and let $z_x(c)=n-\dist(x,c)$ be the number of coordinates in which $c$ agrees with $x$. For $c\ne x$, Theorem~\ref{thm:optimal_weights} gives $z_x(c)\in\{0,1,\ldots,k-3\}\cup\{t\}$, where the first set is empty for $k=2$. Count pairs $(c,J)$ with $c\ne x$ and $J$ a $j$-subset of agreement coordinates. The orthogonal-array property gives
		\begin{equation}\label{eq:moments}
			\sum_{c\ne x}\binom{z_x(c)}j=\binom nj(q^{k-j}-1),
			\qquad 0\le j\le k-1.
		\end{equation}
		Indeed, on each prescribed $j$-set there are exactly $q^{k-j}$ words agreeing with $x$, one of which is $x$ itself. At $j=k-2$, only codewords with $z_x(c)=t$ contribute, so \eqref{eq:moments} gives \eqref{eq:minimum_count} for $A_d(x)$. When $k=2$, this accounts for all $q^2-1$ codewords other than $x$. For $k\ge3$, put $B_j(x)=A_{n-j}(x)$. Taking $j=k-3,\ldots,0$ in descending order, each equation introduces one new unknown $B_j(x)$ with coefficient $\binom{j}{j}=1$. This gives the recursion \eqref{eq:triangular} for the counts about $x$. Since the initial value $A_d(x)$ and the recursion depend only on $q,k,s$, the resulting distribution is independent of $x$.
	\end{proof}

	For each $j=0,\ldots,k-2$, we count agreements with $x$ over all $j$-subsets of coordinate positions. Together, these totals determine the distance distribution. For $j=k-1$, counting agreements separately for each prescribed $(k-1)$-subset $J$ yields a design on the coordinate positions.
	
	\begin{corollary}[Minimum-distance agreement designs]\label{thm:zero_set_design}
		Fix $x\in\mathcal C$ in a length-maximal $(n,q^k,d)_q$ code. For each $c$ with $\dist(x,c)=d$, take the block $Z_x(c)=\{i:c_i=x_i\}$. With blocks counted with multiplicity, these form a $(k-1)$-$(n,s+k-1,q-1)$ design. In dimension three this is a $2$-$(n,s+2,q-1)$ design with replication number $q^2-1$ and $n(q^2-1)/(s+2)$ blocks.
	\end{corollary}
	\begin{proof}
		Every block has size $t=s+k-1$. For a prescribed set of $k-1$ coordinates, exactly $q$ codewords agree with $x$ on them. Apart from $x$, each has at least $k-1$ agreements and hence, by the distance spectrum, exactly $t$ agreements. Thus each $(k-1)$-set lies in exactly $q-1$ blocks. When $k=3$, exactly $q^2-1$ other codewords agree with $x$ in any prescribed coordinate, and each is at distance $d$ from $x$. This gives the replication number, and \eqref{eq:minimum_count} gives the number of blocks.
	\end{proof}
	
	For $s=0$, the code is an orthogonal array of strength $k$ and index one in the sense of Bush~\cite{Bush1952}. At the length $n=q+k-1$, Theorem~\ref{thm:distribution} then gives the classical MDS distance distribution \cite[Theorem~1]{Alderson2020}. For any $s$, a negative or nonintegral count in the prescribed distribution rules out a length-maximal code with those parameters.
	
	\subsection{Spectral integrality and divisibility}
	The following elementary observation uses the standard embedding of Hamming space into Euclidean space. It is related to the spectral approach to two-distance sets of Larman, Rogers, and Seidel \cite{LarmanRogersSeidel1977} and its coding-theoretic applications \cite{BoyvalenkovEtAl2021,BoyvalenkovEtAl2022}. For a code of size $N$ with nonzero distances $d$ and $n$, Boyvalenkov et al.\ prove that
	\[
	N>2n(q-1)+1
	\quad\Longrightarrow\quad
	d=\frac{(m-1)n}{m}
	\quad\text{for some integer }m\ge2
	\]
	\cite[Theorem~9]{BoyvalenkovEtAl2022}. In particular, $d/(n-d)=m-1$, so $(n-d)\mid d$.
	
	Lemma~\ref{lem:two_distance_integrality} obtains the same divisibility conclusion under the weaker hypothesis \(N>n(q-1)+1\) by exhibiting \(d/(n-d)\) as a rational eigenvalue of an integer matrix. Theorem~9 of \cite{BoyvalenkovEtAl2022} gives additional restrictions under its stronger hypothesis.
	
	\begin{lemma}\label{lem:two_distance_integrality}
		Let $\mathcal C$ be a $q$-ary code of length $n$ and size $N$, whose nonzero distances belong to $\{d,n\}$, where $0<d<n$. If $N>1+n(q-1)$, then $(n-d)\mid d$.
	\end{lemma}
	\begin{proof}
		For each $a\in\A$, let $e_a$ be the corresponding standard basis vector of $\mathbb R^q$, and let $\mathbf1_q$ be the all-one vector. Associate to a word $x$ the vector
		\[
		z(x)=\bigl(e_{x_1}-q^{-1}\mathbf1_q,\ldots,e_{x_n}-q^{-1}\mathbf1_q\bigr)\in\mathbb R^{nq}.
		\]
		Each coordinate block has sum zero, so these vectors span a space of dimension at most $n(q-1)$. Their Gram matrix $G$ therefore has rank at most $n(q-1)$. For any $x,y\in\mathcal C$,
		\[
		\langle z(x),z(y)\rangle=\frac{n(q-1)}q-\dist(x,y).
		\]
		Since the only nonzero distances are $d$ and $n$, let $A$ be the symmetric $0$--$1$ matrix with zero diagonal for which $A_{xy}=1$ precisely when $\dist(x,y)=n$. Thus the distance matrix is $d(J-I)+(n-d)A$, and hence
		\[
		G=dI+\left(\frac{n(q-1)}q-d\right)J-(n-d)A,
		\]
		where $I,J$ are the identity and all-one matrices of order $N$. Since $G$ has order $N$ and $\operatorname{rank}G\le n(q-1)$, the rank--nullity theorem and the size assumption give
		\[
		\dim\ker G
		=N-\operatorname{rank}G
		\ge N-n(q-1)>1.
		\]
		Thus $\ker G$ has dimension at least two, and hence contains a nonzero vector $u$ orthogonal to the all-one vector. Thus $Ju=0$ and
		\[
		Au=\frac d{n-d}u.
		\]
		Since $A$ has integer entries, its eigenvalues are algebraic integers. Thus $d/(n-d)$ is a rational algebraic integer and hence belongs to	$\mathbb Z$.
	\end{proof}
	
	\begin{lemma}\label{lem:divisibility}
		If $\mathcal C$ is a length-maximal $(n,q^k,d)_q$ code with $k\ge3$, then $s\le q-1$. If $s\le q-2$, then $(s+2)\mid q$. In particular, $(s+2)\mid q(q+1)$ in all cases.
	\end{lemma}
	\begin{proof}
		Shorten to dimension three, preserving length-maximality and defect, and put $n'=(s+1)(q+1)+1$. The only possible nonzero distances are $d=q(s+1)$ and $n'$. Theorem~\ref{thm:distribution} gives
		\[
		A_{n'}=q^3-1-\frac{n'(q^2-1)}{s+2}
		=\frac{q(q-1)(q-1-s)}{s+2}.
		\]
		Nonnegativity yields $s\le q-1$. If $s\le q-2$, then
		\[
		q^3-1-n'(q-1)=(q^2-1)(q-s-1)>0.
		\]
		This is readily recognised as the condition $q^3>1+n'(q-1)$ in Lemma~\ref{lem:two_distance_integrality}, giving $(s+2)\mid q(s+1)$. Since $\gcd(s+1,s+2)=1$ this gives 	$(s+2)\mid q$. At $s=q-1$, the weaker condition is immediate.
	\end{proof}
	
	The condition $s\le q-2$ is necessary. For example the simplex code $[q^2+q+1,3,q^2]_q$ has $s=q-1$ and $s+2=q+1\nmid q$. The divisibility condition therefore recovers the additive restriction of \cite[Theorem~5]{alderson2026setssubspacesrestrictedhyperplane} throughout its applicable range.
	
	\begin{corollary}[Griesmer equality]\label{cor:lm_griesmer_equality}
		Every length-maximal $(n,q^k,d)_q$ code attains the numerical Griesmer bound:
		\[
		n=g_q(k,d):=\sum_{i=0}^{k-1}\left\lceil\frac{d}{q^i}\right\rceil.
		\]
	\end{corollary}
	\begin{proof}
		Length-maximality gives $d=q(s+1)$ and $n=d+(s+1)+k-2$. For $k=2$ this is $d+\lceil d/q\rceil$. For $k\ge3$, Lemma~\ref{lem:divisibility} gives $d\le q^2$, so all terms of the Griesmer sum with exponent at least two equal one, again giving $n$.
	\end{proof}
	
	\section{Dimension restrictions and exceptional codes}\label{sec:dimensions}
	\subsection{Dimensions three and four}
	For $k=3$, Theorem~\ref{thm:distribution} gives
	\begin{equation}\label{cor:dim23_enumerator}
		A_d=\frac{n(q^2-1)}{s+2},\qquad A_n=q^3-1-A_d,
		\qquad n=(s+1)(q+1)+1.
	\end{equation}
	Only $d$ and $n$ can occur as nonzero distances, and the latter is absent when $s=q-1$. The formula for $A_d$ also counts the blocks of the design in Corollary~\ref{thm:zero_set_design}, but the design contains additional incidence information. These codes also meet the two-distance bound
	\[
	|\mathcal C|\le\frac{q^2d}{qd-(q-1)(n-1)}
	\]
	in \cite[Theorem~1]{BoyvalenkovEtAl2022}  (the denominator equals $s+1$, and the right side is $q^3$).
	
	To deal with the MDS case in higher dimensions, we recall the following result, which applies to arbitrary alphabets. The $L(q,n,k)$ spaces of Bruen and Silverman \cite{BruenSilverman1983} are precisely the $(n,q^k,n-k+1)_q$ MDS codes in our notation.
	
	\begin{theorem}[Bruen--Silverman {\cite[Theorems~4 and~6]{BruenSilverman1983}}]
		\label{thm:bruen_silverman}
		Let $q > 2$ and $k \ge 3$ be integers, and suppose a $(q+k-1,\, q^k,\, q)_q$ MDS code exists.
		\begin{enumerate}
			\item  $4 \mid q$; in particular no such code exists for odd $q$.
			\item If $k \ge 4$, then $36 \mid q$.
		\end{enumerate}
	\end{theorem}

	\begin{remark}\label{rem:bs_nonlinear}\label{rem:s0_mds}
		When $s=0$, a length-maximal code has parameters $(q+k-1,q^k,q)_q$. For $q>2$ and $k\ge4$, Theorem~\ref{thm:bruen_silverman} shows that such a code can exist only if $36\mid q$. The smallest parameters allowed by this divisibility condition occur at $q=36$ and $k=4$, giving $(39,36^4,36)_{36}$.
	\end{remark}
	
	For $k=4$, write $t=s+3$, $n=(s+1)(q+1)+2$, and $d=q(s+1)$. The possible nonzero distances are $d,n-1,n$.
	
	\begin{theorem}\label{thm:dim4_enumerator}
		A length-maximal $(n,q^4,d)_q$ code of defect $s$ satisfies
		\[
		A_d = \frac{n(n-1)(q^2-1)}{(s+3)(s+2)}, \qquad A_{n-1} = n(q^3-1) - (s+3)\,A_d,
		\]
		and the number of neighbours at distance $n$ from any baseword is
		\[
		A_n = \frac{q(q-1)\,(t-q-1)\bigl((q+1)t-(4q+2)\bigr)}{t} \;=\; \frac{q(q-1)\,(s+2-q)\bigl((q+1)s-(q-1)\bigr)}{s+3}.
		\]
	\end{theorem}
	
	\begin{proof}
		Equation~\eqref{eq:minimum_count} gives $A_d$. The first moment \eqref{eq:moments} is $A_{n-1}+tA_d=n(q^3-1)$, giving $A_{n-1}$. Finally $A_n=q^4-1-A_d-A_{n-1}$; substituting $n=(q+1)(t-2)+2$ gives the displayed factorisation.
	\end{proof}
	
	\begin{corollary}\label{cor:dim4_trichotomy}
		Let $\mathcal C$ be a length-maximal $(n,q^4,d)_q$ code. Then one of the following holds:
		\begin{enumerate}
			\item $s=0$, and $36\mid q$;
			\item $s=q-2$ and $A_n=0$;
			\item $q=2$, $s=1$, and $\mathcal C$ has parameters $(8,16,4)_2$ and distribution $1+14z^4+z^8$.
		\end{enumerate}
		In particular, for $q>2$ every positive-defect code has $s=q-2$.
	\end{corollary}
	\begin{proof}
		Lemma~\ref{lem:divisibility} gives $s\le q-1$. For $s\ge1$, the factor $(q+1)s-(q-1)$ in Theorem~\ref{thm:dim4_enumerator} is positive. Nonnegativity of $A_n$ therefore excludes $1\le s\le q-3$, leaving $s\in\{0,q-2,q-1\}$. The MDS case follows from Theorem~\ref{thm:bruen_silverman}; when $q=2$, we have $s=0=q-2$, so this is included in case~(2). Substitution gives $A_n=0$ when $s=q-2$.
		
		It remains to consider $s=q-1$. Here $n=q^2+q+2$, $d=q^2$, and Theorem~\ref{thm:dim4_enumerator} gives
		\[
		A_{n-1}=n\bigl((q^3-1)-(q^2+q+1)(q-1)\bigr)=0.
		\]
		Since $ q^4-1-n(q-1)=(q-1)(q^3-1)>0$, Lemma~\ref{lem:two_distance_integrality} yields $(q+2)\mid q^2$, or equivalently $(q+2)\mid4$. Thus $q=2$, and substitution gives the distribution in case~(3).
	\end{proof}
	
	For every prime power $q$, an ovoid of $\mathrm{PG}(3,q)$ gives a linear code in case~(2), with parameters $(q^2+1,q^4,q^2-q)_q$ \cite[Section~16.1]{Hirschfeld1985}. The corollary extends the positive-defect restriction established for additive codes by Alderson and Ball \cite[Theorem~7 and Example~17]{alderson2026setssubspacesrestrictedhyperplane} to arbitrary alphabets. Case~(3) is unique up to equivalence, as shown below.
	
	We recall the two cases of Rao's bound needed here.
	\begin{theorem}[Rao {\cite[p.~129]{Rao1947}}; see also {\cite[Theorem~2.1]{HedayatSloaneStufken1999}}]\label{thm:rao}
		An $\operatorname{OA}(N,n,q,3)$ satisfies
		\[
		N\ge q\bigl(1+(n-1)(q-1)\bigr).
		\]
		An $\operatorname{OA}(N,n,q,4)$ satisfies
		\[
		N\ge1+n(q-1)+\binom n2(q-1)^2.
		\]
		An array attaining the corresponding bound is called tight.
	\end{theorem}
	
	\begin{remark}\label{rem:dim4_noda}
		The exclusion of $s=q-1$ for $q>2$ also follows from Noda's classification \cite[Theorem~1]{Noda1986}. The parameters would give an $\operatorname{OA}(q^4,q^2+q+2,q,3)$, attaining Rao's bound because
		\[
		q\bigl(1+(q^2+q+1)(q-1)\bigr)=q^4.
		\]
		Noda proves that a tight strength-three array has parameters either $(2n,n,2,3)$, with $4\mid n$, or $(q^3,q+2,q,3)$, with $q$ even. The proposed parameters fit neither family for $q>2$. 
	\end{remark}
	
	\subsection{Nonexistence in dimension five and above}
	Shortening reduces the positive-defect problem to the cases in Corollary~\ref{cor:dim4_trichotomy}. Rao's strength-four bound then gives the obstruction in dimension five.
	
	\begin{theorem}[High-dimensional nonexistence]\label{thm:dim5existence}\label{thm:dim5}
		Let $\mathcal C$ be a length-maximal $(n,q^k,d)_q$ code with $k\ge5$. Then either
		\begin{enumerate}
			\item $s=0$ and either $q=2$ or $36\mid q$; or
			\item $s=1$ and $q=3$.
		\end{enumerate}
		Consequently, for every $q\ge4$ no non-MDS length-maximal code has dimension at least five. For every prime-power $q\ge4$, no length-maximal code has dimension at least five.
	\end{theorem}
	\begin{proof}
		Successive shortening gives length-maximal codes of dimensions four and five with the same defect. If $s=0$ and $q>2$, Theorem~\ref{thm:bruen_silverman} requires $36\mid q$.
		
		Suppose $s>0$. For $q=2$, Corollary~\ref{cor:dim4_trichotomy} forces $s=1$. The dimension-five shortening would be an $\operatorname{OA}(32,9,2,4)$, contradicting Rao's bound $32\ge1+9+\binom92=46$.
		
		For $q>2$, the same corollary forces $s=q-2$. The dimension-five shortening is then an $\operatorname{OA}(q^5,q^2+2,q,4)$. The difference between Rao's lower bound and its number of rows is
		\[
		1+(q^2+2)(q-1)+\binom{q^2+2}{2}(q-1)^2-q^5
		=\frac{q^2(q-1)(q-3)(q^2+1)}2.
		\]
		This is positive for $q>3$, so only $q=3$ and $s=1$ remains. No prime power is divisible by $36$, giving the final assertion.
	\end{proof}
	
	The distance distribution gives a second proof of the exclusion for $q>3$. When $k=5$ and $s=q-2$, Theorem~\ref{thm:distribution} gives
	\[
	A_n=-\frac{q^3(q-1)^2(q-3)}{2(q+2)}.
	\]
	For $q>3$ this is negative, whereas $A_n$ counts the codewords at distance $n$ from a fixed codeword. This is the combinatorial counterpart of the negative secant count in the additive argument of \cite[Lemma~6]{alderson2026setssubspacesrestrictedhyperplane}.
	
	When $q=3$, the dimension-five parameters attain equality in Rao's bound, giving a tight $\operatorname{OA}(3^5,11,3,4)$. Noda \cite{Noda1979} determined the possible parameters of tight arrays of strength four, as summarized in \cite[Theorem~1.1]{GavrilyukSudaVidali2019}. Gavrilyuk, Suda, and Vidali completed the classification by ruling out the remaining family with $q=6$ \cite[Corollary~4.2]{GavrilyukSudaVidali2019}. The dimension-six ternary exception considered below likewise attains Rao's bound, giving the tight $\operatorname{OA}(3^6,12,3,5)$ in Noda's classification \cite[Theorem~2]{Noda1986}.
	
	\subsection{The binary and ternary exceptions}
	\begin{proposition}\label{prop:dim6_q3}\label{prop:q3_dimension}
		For $q=3$ and $k\ge5$, a length-maximal code has $s=1$ and $k\in\{5,6\}$. Its distance enumerator about any baseword is
		\[
		\begin{array}{ll}
			k=5:&1+132z^6+110z^9,\\
			k=6:&1+264z^6+440z^9+24z^{12}.
		\end{array}
		\]
		Both cases exist, as the dual and extended ternary Golay codes, respectively.
	\end{proposition}
	\begin{proof}
		Theorem~\ref{thm:dim5existence} forces $s=1$, so $n=k+6$ and $d=6$. Substitution in Theorem~\ref{thm:distribution} gives the two enumerators. At $k=7$, equation~\eqref{eq:minimum_count} would require $A_6=8 {\binom{13}5}/{\binom75}=\frac{3432}{7},$ which is not an integer. Shortening excludes every larger dimension. The stated Golay codes have these parameters \cite{DelsarteGoethals1975}; see also \cite[Chapter~20]{MacWilliamsSloane1977}.
	\end{proof}
	
	\begin{theorem}[Classification in dimensions at least five]\label{thm:high_classification}\label{cor:dim5_linearity}
		Let $\mathcal C$ be length-maximal with $k\ge5$. If $36\nmid q$, then $\mathcal C$ is equivalent to one of
		\[
		[k+1,k,2]_2,\qquad [11,5,6]_3,\qquad [12,6,6]_3,
		\]
		the binary even-weight code, the dual ternary Golay code, and the extended ternary Golay code, respectively. Each is unique up to equivalence. If $36\mid q$, then $\mathcal C$ must be MDS, and no existence assertion is made in that case.
	\end{theorem}
	\begin{proof}
		The parameter restrictions follow from Theorem~\ref{thm:dim5existence} and Proposition~\ref{prop:dim6_q3}. Every binary MDS code is equivalent to linear \cite[Theorem~4]{Alderson2020}, giving the even-weight code.
		
		For the ternary cases, normalise a word to zero and identify the alphabet with $\F_3$. All weights and distances in Proposition~\ref{prop:dim6_q3} are divisible by three. Since $\wt(u)\equiv u\cdot u\pmod3$, $ \wt(u-v)\equiv\wt(u)+\wt(v)-2u\cdot v\pmod3 $ shows that $u\cdot v=0$ for every $u,v\in\mathcal C$. The linear span is therefore self-orthogonal and has dimension at most $\lfloor n/2\rfloor$, namely five or six. Since $|\mathcal C|=3^{\lfloor n/2\rfloor}$, the code equals its span.
		
		For length eleven, the code has exactly two nonzero distances and $		|\mathcal C|=3^5=1+11(3-1)+\binom{11}{2}(3-1)^2$. It therefore attains the generalized Hadamard bound of Delsarte and Goethals \cite[Equation~(2.2)]{DelsarteGoethals1975} and is a generalized Hadamard code of order two. Their Theorem~4.4(i) identifies it with the expurgated ternary Golay code, that is, the dual of the perfect ternary Golay code.
		
		For length twelve, puncturing any coordinate gives $3^6$ words with minimum distance at least five, and since $3^6\left(1+11\cdot2+\binom{11}{2}2^2\right)=3^{11}$, the punctured code is perfect. Thus the length-twelve code is an extended perfect two-error-correcting code, and Theorem~4.4(ii) of the same paper identifies it with the extended ternary Golay code.
	\end{proof}
	
	For completeness, we show the dimension-four binary case with $s=1$ is also unique.
	\begin{proposition}\label{prop:binary_unique}
		Every $(8,16,4)_2$ code is equivalent to the extended binary Hamming code.
	\end{proposition}
	\begin{proof}
		Normalise a word to zero, and puncture one coordinate. The code is length-maximal before puncturing, and Corollary~\ref{cor:dim4_trichotomy} gives fourteen words of weight four and one word of weight eight, in addition to the zero word. Symbol-uniformity gives eight entries equal to one in the deleted coordinate. The weight-eight word accounts for one of them, so exactly seven weight-four words have a one there. The punctured distribution is therefore $1+7z^3+7z^4+z^7$.
		
		If $A$ and $B$ are supports of two of the seven weight-three words, then the punctured minimum distance gives $6-2|A\cap B|\ge3$, and hence $|A\cap B|\le1$. These seven triples contain $7\binom32=21=\binom72$ pairs in total, so every pair occurs exactly once. They form the unique Steiner triple system on seven points, the Fano plane. The all-one word is present, and the original code is closed under complementation since each word has a unique neighbour at distance eight. Hence the four-element supports in the punctured code are the complements of the triples. The punctured code is therefore the Hamming code. To restore the deleted coordinate, the new bits must differ on every pair at distance three. This graph on the Hamming code is connected because its weight-three words span the code, so the extension is unique up to complementing the new coordinate. It is the overall-parity extension.
	\end{proof}
	
	\begin{corollary}[Existence over prime-power alphabets]
		\label{cor:prime_power_existence}
		Let $q$ be a prime power. In each of the following cases, the listed values are precisely the Singleton defects attained by length-maximal $q$-ary codes.
		\begin{enumerate}
			\item In dimension four, the defect set is $\{q-2\}$ for $q>2$ and $\{0,1\}$ for $q=2$.
			
			\item In dimension three with $q=2^e$, 	the defect set is $\{2^i-2:1\le i\le e\}\cup\{q-1\}$.

			\item In dimension three with $q$ an odd prime, the defect set is $\{q-2,q-1\}$.
		\end{enumerate}
	\end{corollary}
	\begin{proof}
		The necessity in part 1 follows from Corollary~\ref{cor:dim4_trichotomy}, since no prime power is divisible by $36$. Ovoid codes give $s=q-2$, including the binary $[5,4,2]_2$ code; the extended binary Hamming code gives $(q,s)=(2,1)$.
		
		For dimension three, Lemma~\ref{lem:divisibility} gives $s\le q-1$ and requires $(s+2)\mid q$ when $s\le q-2$. These are exactly the stated necessary conditions in parts 2 and 3. At $s=q-1$, the $[q^2+q+1,3,q^2]_q$ simplex code gives existence. At $s=q-2$, shorten an ovoid code once, and Theorem~\ref{thm:optimal_uniform} gives a length-maximal code of dimension three with the same defect.
		
		For the remaining defects in part~2, let $r=2^i$ with $1\le i<e$. Denniston's maximal arc of degree $r$ in $\mathrm{PG}(2,q)$ \cite{Denniston1969} has $n=q(r-1)+r$ points and meets each line in zero or $r$ points. Since $n>r$, the arc spans the plane, and its points give a linear $[n,3,n-r]_q$ code. Its defect is $s=r-2$, and $n=(s+1)(q+1)+1$, so the code is length-maximal.
	\end{proof}
	
	Table~\ref{tab:by_dimension} collects, by dimension, the necessary conditions for positive-defect length-maximal codes obtained so far, together with the known examples.
	
	\begin{table}[!htb]
		\centering\small
		\caption{Necessary conditions for length-maximal codes of positive defect. All distributions are determined by Theorem~\ref{thm:distribution}.}\label{tab:by_dimension}
		\begin{tabularx}{\linewidth}{@{}c X X@{}}\toprule
			$k$ & Necessary conditions & Examples and status\\\midrule
			2 & No upper bound on $s$ & Resolvable $2$-$(q^2,q,s+1)$ multidesigns; all $s$ occur for prime-power $q$.\\
			3 & $s\le q-1$; $(s+2)\mid q$ if $s\le q-2$ & Maximal-arc and simplex constructions; Mathon's $(21,9^3,18)_9$ example.\\
			4 & $s=q-2$, or $(q,s)=(2,1)$ & Ovoid codes at $s=q-2$; the binary extended Hamming code at $(q,s)=(2,1)$.\\
			5,6 & $q=3$, $s=1$ & Unique dual and extended ternary Golay codes.\\
			$\ge7$ & Impossible & No positive-defect code.\\\bottomrule
		\end{tabularx}
	\end{table}

	\section{Further bounds and Griesmer consequences}\label{sec:refinements}
	
	Corollary~\ref{cor:lm_griesmer_equality} shows that every length-maximal code attains the numerical Griesmer bound. We now ask under what conditions an arbitrary $(n,q^k,d)_q$ code satisfies the inequality
	\[
	n\ge g_q(k,d), \qquad
	g_q(k,d)=\sum_{i=0}^{k-1}\left\lceil\frac{d}{q^i}\right\rceil.
	\]
	The inequality holds for linear codes \cite{Griesmer1960,SS65}, but nonlinear codes can violate it \cite{Bellini2015Griesmer,Guerrini2015Optimal}.
	
	Several related bounds are known. Bellini, Guerrini, and Sala \cite{BelliniGuerriniSala2014} studied bounds on the size of unrestricted, systematic, and linear codes. Ball, Lavrauw, and Popatia \cite{Ball2025Griesmer} obtained Griesmer-type bounds for additive codes, and Pan et al.\ \cite{PanEtAl2025} studied the minimum weights of complementary codes of linear codes.
	
	Our approach is to shorten to a code of size $q^r$ and apply the Plotkin bound. Varying $r$ gives a hierarchy of length bounds from which we derive conditions on the minimum distance or the Singleton defect that suffice for the Griesmer inequality to hold. In the binary case, shortening to eight words allows us to sharpen these conditions using the bound of Guerrini, Meneghetti, and Sala \cite[Corollary~32]{Guerrini2015Optimal}.

	\subsection{A hierarchy of shortening bounds}\label{sec:large_defect}
	\begin{theorem}[Shortening--Plotkin hierarchy]\label{thm:plotkin_shortening_hierarchy}
		Let $\mathcal C$ be an $(n,q^k,d)_q$ A$^s$MDS code with $k\ge2$. For every integer $r$ with $2\le r\le k$,
		\begin{equation}\label{eq:plotkin_shortening_hierarchy}
			n\le
			\left\lfloor (s+r-1)\frac{q^r-1}{q^{r-1}-1}\right\rfloor+k-r.
		\end{equation}
		Equivalently, each shortening depth gives
		\begin{equation}\label{eq:hierarchy_distance}
			n\ge k-r+\left\lceil d\sum_{i=0}^{r-1}q^{-i}\right\rceil.
		\end{equation}
	\end{theorem}
	
	\begin{proof}
		Using Lemma~\ref{lem:shortening}, shorten $k-r$ times to a code of length $N=n-k+r$, size $q^r$, and minimum distance $d'\ge d$. Its defect is $s'=N-r+1-d'=n-k+1-d'\le s$. The Plotkin inequality \eqref{eq:plotkin_length} gives
		\[
		N\ge d'\frac{1-q^{-r}}{1-q^{-1}}
		=(N-r+1-s')\frac{q^r-1}{q^{r-1}(q-1)}.
		\]
		Rearranging and using $s'\le s$ yields
		\[
		N\le (s'+r-1)\frac{q^r-1}{q^{r-1}-1}
		\le (s+r-1)\frac{q^r-1}{q^{r-1}-1}.
		\]
		Since $N$ is integral, taking the floor and adding $k-r$ gives \eqref{eq:plotkin_shortening_hierarchy}. The Plotkin bound \eqref{eq:plotkin_length} and $d'\ge d$ also give $N\ge d\sum_{i=0}^{r-1}q^{-i}$. Taking the ceiling and adding $k-r$ gives \eqref{eq:hierarchy_distance}.
	\end{proof}

	Taking $r=2$ in \eqref{eq:plotkin_shortening_hierarchy} recovers the bound \eqref{eqn:maximal_arc_bound}, of  Theorem~\ref{thm:main_bound}, whereas for $k\ge3$, taking $r=3$ gives
	\begin{equation}\label{cor:dim3_plotkin_floor}
		n\le q(s+2)+\left\lfloor\frac{s+2}{q+1}\right\rfloor+k-3.
	\end{equation}
	Before rounding, the upper bound \eqref{cor:dim3_plotkin_floor} minus the upper bound \eqref{eqn:maximal_arc_bound} is  $\frac{q(q-1-s)}{q+1}$.  	Since the $r=2$ bound is integral, \eqref{cor:dim3_plotkin_floor} 	is strictly stronger precisely when $s\ge q$. The two integer bounds agree for $s=q-2$ and $s=q-1$.

	\begin{remark}
		For $k=3$, the improved bound is sharp for infinitely many defects over every prime-power alphabet. Indeed, repeating each coordinate of the $[q^2+q+1,3,q^2]_q$ simplex code $m$ times gives a code of length $m(q^2+q+1)$, minimum distance $mq^2$, and defect $s=m(q+1)-2$. Since $(s+2)/(q+1)=m$ is an integer, this code attains \eqref{cor:dim3_plotkin_floor} with equality even before rounding, whereas the $r=2$ bound exceeds its length by $q(m-1)$.
		
		More generally, the agreement count of Lemma~\ref{lem:2d_bound} shows that a code of size $q^3$ attains the unrounded $r=3$ bound if and only if it is symbol-uniform and equidistant. Its coordinate partitions then form a resolvable $2$-$(q^3,q^2,s+2)$ multidesign \cite[Theorem~1]{SemakovZinoviev1968}.
	\end{remark}
	
	\subsection{Griesmer consequences over arbitrary alphabets}
	\begin{corollary}\label{cor:griesmer_d_small}\label{cor:griesmer_s_small}
		Every $(n,q^k,d)_q$ code with $k\ge2$ satisfies the Griesmer bound if $k=2$ or $d\le q^2$. In particular, this holds whenever $s\le q-1$.
	\end{corollary}
	\begin{proof}
		Theorem~\ref{thm:main_bound} gives $n\ge d+\lceil d/q\rceil+k-2$. For $d\le q^2$, every term $\lceil d/q^i\rceil$ with $i\ge2$ equals one, so the right side is $g_q(k,d)$. For $k=2$ the same conclusion holds for all $d$. Finally, \eqref{eqn:maximal_arc_bound} gives $d\le q(s+1)$, so $s\le q-1$ implies $d\le q^2$.
	\end{proof}
	
	This extends the range $d\le2q$ for systematic codes in \cite[Theorem~9]{Guerrini2015Optimal} to unrestricted codes and to $d\le q^2$. The inequality \eqref{eqn:singleton_to_griesmer} also gives $n\ge k+\lceil3d/2\rceil-2$ for binary codes, extending the systematic bound in \cite[Proposition~22]{Guerrini2015Optimal} without the systematic assumption.
	
	\begin{corollary}\label{cor:2d_LM_GB}
		For prime-power $q$ and every positive integer $d$, the minimum lengths of unrestricted, systematic, and linear codes of size $q^2$ and minimum distance at least $d$ all equal $g_q(2,d)$.
	\end{corollary}
	\begin{proof}
		The lower bound is Corollary~\ref{cor:griesmer_d_small}. For attainment, put $m=\lceil d/q\rceil$ and $a=mq-d$, so $0\le a\le q-1$. Take $m$ copies of each column of a generator matrix of the $[q+1,2,q]_q$ simplex code and delete $a$ columns while leaving some projective direction with multiplicity $m$. This is possible since $a\le q-1$ and there are $q+1$ directions. At least two directions remain, so the dimension is two; the largest direction multiplicity is $m$, giving length $m(q+1)-a=d+m$ and minimum distance $d$. A linear code is equivalent to a systematic one, proving the assertion.
	\end{proof}
	
	The next consequence retains a useful defect-dependent range beyond $d=q^2$.
	\begin{corollary}\label{cor:large_defect_griesmer}
		Let $k\ge3$ and $1\le\beta\le q$ be integers. If $s+2\ge\beta(q+1)$ and $d\le\beta q^2$, then an $(n,q^k,d)_q$ code satisfies the Griesmer bound. In particular, if $s\ge2q$ and $d\le2q^2$, then it satisfies the Griesmer bound and $n\le s(q+1)+k-1$.
	\end{corollary}
	\begin{proof}
		Write $s+2=\beta'(q+1)+r$, where $\beta'=\lfloor(s+2)/(q+1)\rfloor\ge\beta$ and $0\le r\le q$. Equation~\eqref{cor:dim3_plotkin_floor} gives
		\begin{equation} \label{eqn:LDG_1}
			n\le(s+2-\beta')(q+1)+k-3+\beta'-r
			\le(s+2-\beta)(q+1)+k-3+\beta.
		\end{equation}
		Substituting $s=n-k+1-d$ and rearranging yields
		\begin{equation}\label{eqn:LDG_2}
			n\ge d+\left\lceil\frac dq\right\rceil+\beta+k-3.
		\end{equation}
		Since $d\le\beta q^2\le q^3$, the term $\lceil d/q^2\rceil$ is at most $\beta$, and every term with exponent at least three equals one. Thus the lower bound \eqref{eqn:LDG_2} is at least $g_q(k,d)$. For the final length assertion, take $\beta=2$ in the upper bound \eqref{eqn:LDG_1}.
	\end{proof}
	
	\subsection{Binary bounds from eight-word codes}\label{sec:binary_refinements}
	Guerrini, Meneghetti, and Sala \cite[Corollary~32]{Guerrini2015Optimal} proved that every binary code with eight words satisfies the Griesmer bound. In the present setting, the required inequality also follows directly from the Plotkin bound, together with a parity extension when the minimum distance is odd.
	
	\begin{proposition}\label{prop:binary_three_terms}
		Every binary $(n,2^k,d)_2$ code with $k\ge3$ satisfies
		\begin{equation}\label{eq:binary_three_terms}
			n\ge d+\left\lceil\frac d2\right\rceil+\left\lceil\frac d4\right\rceil+k-3.
		\end{equation}
	\end{proposition}
	
	\begin{proof}
		Shorten $k-3$ times to obtain a binary code with eight words, length $N=n-k+3$, and minimum distance $d'\ge d$. It is enough to prove $	N\ge d+\left\lceil\frac d2\right\rceil +\left\lceil\frac d4\right\rceil.$
		
		If $d$ is even, the Plotkin bound gives $N\ge\left\lceil\frac{7d}{4}\right\rceil =d+\left\lceil\frac d2\right\rceil +\left\lceil\frac d4\right\rceil.$ If $d$ is odd, append an overall parity coordinate. The resulting eight-word code has length $N+1$ and minimum distance at least $d+1$. Applying the Plotkin bound gives $N+1\ge\left\lceil\frac{7(d+1)}4\right\rceil, \text{ and hence } N\ge d+\left\lceil\frac d2\right\rceil +\left\lceil\frac d4\right\rceil.$ The stated inequality follows upon substituting $N=n-k+3$.
	\end{proof}
	
	\begin{corollary}\label{cor:binary_d8_griesmer}\label{cor:binary_refinement_general}
		A binary $(n,2^k,d)_2$ code with $k\ge3$ has the following properties.
		\begin{enumerate}
			\item It satisfies the Griesmer bound whenever $d\le8$.
			\item For every integer $j\ge0$, if $s\ge3j+2$, then
			\begin{equation}\label{eq:binary_defect_refinement}
				n\le3s+k-1-2j.
			\end{equation}
		\end{enumerate}
	\end{corollary}
	\begin{proof}
		For $d\le8$, one has $\lceil d/2^i\rceil=1$ for every $i\ge3$, so \eqref{eq:binary_three_terms} gives the Griesmer bound.
		
		For the second assertion, substitute $n=d+s+k-1$ in \eqref{eq:binary_three_terms} to obtain
		
		\begin{equation} \label{eq: bin_GB_1}
			\left\lceil\frac d2\right\rceil+
			\left\lceil\frac d4\right\rceil\le s+2.
		\end{equation}
		Fix an integer $j\ge0$ with $s\ge3j+2$, and suppose $n\ge3s+k-2j$. Then  $d=n-k+1-s\ge2s+1-2j$, and $2s+1-2j\ge4j+5$, giving contradicting \eqref{eq: bin_GB_1}. The result follows.
	\end{proof}
	
	We note that the Nordstrom--Robinson code $(16,2^8,6)_2$ has $s=3$ and meets both \eqref{eq:binary_three_terms} and the case $j=0$ of \eqref{eq:binary_defect_refinement} with equality \cite{NordstromRobinson1967}.
	
	\subsection{AMDS codes close to the length bound}
	\begin{proposition}\label{lem:binary_amds_bounds}\label{lem:qary_amds_hamming}
		If $\mathcal C$ is an $(n,q^k,d)_q$ AMDS code with $k\ge2$ and $n\ge k+2q-1$, then
		\[
		\sum_{i=0}^{q-1}\binom{k+2q-1}{i}(q-1)^i\le q^{2q-1}.
		\]
		Consequently $k\le4$ for $q=2$, and $k\le6$ for $q=3$. In particular, a binary AMDS code with $k\ge5$ has $n\le k+2$.
	\end{proposition}
	\begin{proof}
		Put $N=k+2q-1$. Puncturing $n-N$ coordinates decreases each pairwise distance by at most $n-N$. Since $d=n-k$, the punctured code still has $q^k$ words and minimum distance at least $d-(n-N)=2q-1$. The Hamming balls of radius $q-1$ about these words are therefore disjoint, so the Hamming bound \cite{Hamming1950} gives
		\[
		q^k\sum_{i=0}^{q-1}\binom{N}{i}(q-1)^i\le q^N.
		\]
		Dividing by $q^k$ and using $N-k=2q-1$ proves the inequality. For $q=2$, it reduces to $k+4\le8$, giving $k\le4$. For $q=3$, it reduces to $1+2(k+5)^2\le3^5$, giving $k\le6$. Thus a binary AMDS code with $k\ge5$ must have $n<k+3$, or equivalently $n\le k+2$.
	\end{proof}
	
	These conclusions apply in particular to length-maximal AMDS codes. The binary examples of dimensions two, three, and four can be obtained by shortening the extended Hamming code, while the ternary dimension-six bound is attained by the extended Golay code.
	
	\section{Open problems}\label{sec:conclusion}
	The equality structure and classical orthogonal-array bounds confine the positive-defect existence question to dimensions at most four, apart from the two ternary Golay codes (see Table~\ref{tab:by_dimension}). The divisibility and dimension-four restrictions now agree with those in the additive setting, but existence and equivalence questions remain.
	
	\paragraph{Dimension three.}
	Corollary~\ref{cor:prime_power_existence} settles parameter existence over even prime-power and odd prime alphabets. For an odd prime power $q=p^e$ with $e\ge2$, the remaining intermediate possibilities have
	\[
	s=p^i-2,\qquad 1\le i<e.
	\]
	Not all of these cases are open. The additive $(21,9^3,18)_9$ code from the perp-system of De Clerck et al.\ establishes existence at $q=9$, $s=1$ \cite{DeClerck2002PerpSystems}; see also \cite[Example~10]{alderson2026setssubspacesrestrictedhyperplane}. Which of the other intermediate parameter sets admit codes? Linear examples throughout this range are excluded by \cite[p.~31]{Ball1997MaximalArcs}. A separate question, including at parameters where existence is known, is whether codes not equivalent to additive codes occur.
	
	\paragraph{Dimension four.}
	Corollary~\ref{cor:prime_power_existence} determines which defects occur for length-maximal codes of dimension four over prime-power alphabets. For $q>2$, the only positive defect is $s=q-2$, and ovoids in $\mathrm{PG}(3,q)$ give linear codes with parameters $(q^2+1,q^4,q^2-q)_q$. For which prime powers $q>2$, if any, do these parameters also admit a genuinely nonlinear code? Every code with these parameters has just two nonzero distances, $d=q^2-q$ and $n-1=q^2$. Can the two distance relations and the minimum-distance agreement designs be used to recognize the ovoid codes up to equivalence? For alphabets of non-prime-power size, the first question is whether codes with these parameters exist at all.

	\paragraph{Extension and near-extremal codes.}
	Extension questions for additive codes motivate the corresponding problem for unrestricted codes \cite{Alderson2026AdditiveExtensions}. An \emph{extension} of an $(n,q^k,d)_q$ code $\mathcal C$ is an $(n+1,q^k,d+1)_q$ code obtained by appending one symbol to each word of $\mathcal C$. The \emph{distance-$d$ graph} has vertex set $\mathcal C$, with distinct words adjacent precisely when their distance is $d$.
	
	The code $\mathcal C$ is extendable if and only if its distance-$d$ graph is properly colourable with at most $q$ colours \cite[Lemma~3.2]{Alderson2026AdditiveExtensions}. The colour assigned to a codeword specifies its appended symbol, since pairs originally at distance $d$ must receive different symbols, while all other pairs already have distance at least $d+1$. Thus a code of length $(s+1)(q+1)+k-3$ and defect $s$ extends to a length-maximal code precisely when this colouring condition holds. Even when the original code is additive, an extension need not preserve additivity \cite{Alderson2026AdditiveExtensions}, so the unrestricted problem is genuinely broader than its linear or additive counterpart.
	
	For $k\ge4$ and $1\le s\le q-3$, the nonexistence results of this paper show that no such extension is possible. It remains to characterize extendability in the admissible parameter ranges and, when extensions exist, to determine whether they are unique up to code equivalence. The dimension-three case is a natural starting point.
	
	The AMDS sphere-packing argument leaves a related boundary at length $k+2q-2$ and distance $2q-2$. For $q=2$, adjoining one constant coordinate to an even-weight code gives examples in every dimension. For $q=3$, however, the Hamming bound still applies. At length $k+4$ and distance four it gives $1+2(k+4)\le3^4$, hence $k\le36$. More generally, for $q\ge3$ the radius-$q-2$ Hamming bound at this length gives $\sum_{i=0}^{q-2}\binom{k+2q-2}{i}(q-1)^i\le q^{2q-2}$ and therefore bounds the dimension for each fixed alphabet. Determining the maximal length of positive-defect codes just below \eqref{eqn:maximal_arc_bound} remains of interest.
	\paragraph{The exceptional MDS possibility.}
	For $q>2$ and $k\ge4$, Bruen--Silverman leaves only $36\mid q$ for a length-maximal MDS code. Such alphabets are necessarily of non-prime-power size. Existence in dimension four, beginning with $(39,36^4,36)_{36}$, is unresolved by the present methods. Any example in higher dimension would shorten to one in dimension four. 
	
	\begingroup
	\small
	\bibliographystyle{plain}
	\bibliography{references}
	\endgroup
\end{document}